\documentclass[12pt]{amsart}
\usepackage{amsmath}
\usepackage{tikz}
\usetikzlibrary{calc} 
\usepackage{amsthm}
\usepackage{amssymb}
\usepackage{graphicx}
\usepackage{blindtext}
\usepackage{hyperref}
\numberwithin{equation}{section}

\usepackage[margin=2.75cm]{geometry}
\usepackage{color}
\usepackage{tabularx}
\usepackage{geometry}
\usepackage{tikz}
\usepackage[mathscr]{euscript}
\usepackage{amsthm}
\usepackage[utf8]{inputenc}
\usepackage{hyperref}
\usepackage{booktabs}
\numberwithin{equation}{section}
\newcommand{\R}{\mathbb{R}}

\newtheorem{theorem}{Theorem}[section]

\newtheorem{lemma}[theorem]{Lemma}
\newtheorem{proposition}[theorem]{Proposition}

\newtheorem{definition}[theorem]{Definition}
\newtheorem{example}[theorem]{Example}
\newtheorem{remark}[theorem]{Remark}
\date{\today}

\date{}
\title{Fractional Torsional Rigidity of Compact Metric Graphs}

\author{Sedef ÖZCAN}
\address{Sedef Özcan, Dokuz Eylül University, Faculty of Science, Department of Mathematics, Izmir, Turkey}
\keywords{Fractional Laplace operator, Quantum Graphs, Torsion Function, Torsional Rigidity}

\subjclass[2020]{Primary: 34B45. Secondary: 05C50, 35P15, 1Q35}

\begin{document}
	
	\maketitle
    \begin{abstract}
        This paper investigates fractional torsional rigidity on compact, connected metric graphs, a novel extension of the classical concept to nonlocal operators. The fractional torsional rigidity is defined as the $L^1$-norm of the fractional torsion function, which is the unique solution to the boundary value problem $(-\Delta_{\mathcal{G}})^\alpha u_\alpha = 1$ on a graph $\mathcal{G}$ with zero boundary conditions at Dirichlet vertices. We establish a variational characterization for this quantity, which serves as a powerful tool to prove a series of results on its geometric dependence. By applying surgery principles, we derive explicit upper and lower bounds, indicating that the interval serves as an upper comparison case and the flower graph as a lower one among graphs of fixed total length. These findings mirror the classical case, yet the methods required are substantially different due to the nonlocal nature of the fractional Laplacian.
    \end{abstract}

	\section{Introduction}
This paper investigates fractional torsional rigidity and the associated torsion function on compact, connected metric graphs. The concept of torsional rigidity has its roots in elasticity theory, where it measures how resistant a body is to twisting forces. From a mathematical point of view, torsional rigidity is defined through a Poisson problem on a domain and has long been studied in relation to isoperimetric inequalities, spectral geometry, and functional inequalities. It provides a natural link between the geometry of a domain and the analytic properties of the Laplacian.

Classically, the \emph{torsional rigidity} of a domain $\Omega \subset \mathbb{R}^2$ is defined as the $L^{1}$-norm of the torsion function
\begin{equation}
	T(\Omega):= \lVert \upsilon \rVert_{L^{1}(\Omega)},
\end{equation}
where $\upsilon$ is the unique solution of the boundary value problem
\[
\begin{cases}
	-\Delta \upsilon(x) =1, &x \in \Omega,\\
	\upsilon(z)=0, & z \in \partial \Omega,
\end{cases}
\]
This quantity was first introduced in mechanics by Saint–Venant to describe the resistance of a cylindrical beam to torsion, and later reinterpreted by Pólya as a purely geometric constant depending only on the shape and size of $\Omega$ \cite{pol1}.
He proved that among all open bounded domains of fixed measure, the disk maximizes torsional rigidity, thereby establishing one of the first isoperimetric inequalities for this quantity.
Following this work, torsional rigidity has been widely investigated in analysis and spectral geometry. Brasco \cite{brasco} applied rearrangement techniques to obtain sharp functional inequalities involving torsional rigidity.

Over the past few decades, metric graphs have emerged as a natural framework to model networks and thin structures such as nanotubes or waveguides \cite[Chapter~7]{BK}.  
Their spectral theory has been studied in great depth \cite{BK,kost1,kost2,kuch1,kuch2}, and more recently, the focus has shifted toward spectral geometry, investigating how the shape and connectivity of a graph influence the spectrum of its Laplacian.  
In this setting, torsional rigidity plays a role analogous to that of the first Laplace eigenvalue in the classical domain case, reflecting fundamental principles such as the Faber--Krahn inequality \cite{Faber-23,Krahn-25}.  
Yet, despite these parallels, torsional rigidity on metric graphs has received little attention so far, with only a few recent contributions available in the literature \cite{colla,MugnoloPlumer2023,OzcanTaeufer2024}.

The first systematic study of torsional rigidity on metric graphs was carried out by Mugnolo and Plümer \cite{MugnoloPlumer2023}, who developed a theory under the assumption that at least one Dirichlet vertex is present.  
This requirement is crucial, since in the absence of Dirichlet vertices, the Laplacian is not invertible, and hence the torsion function cannot be defined or expected to be positive.  
More recently, in \cite{OzcanTaeufer2024}, torsional rigidity for metric graphs equipped with $\delta$-type vertex conditions is investigated.  

In parallel with these developments, increasing attention has been devoted to fractional Laplace operators, motivated by their broad range of applications in physics, probability, and analysis.  
For $\alpha \in (0,1)$, the fractional Laplacian $(-\Delta)^\alpha$ provides a natural interpolation between local diffusion and nonlocal effects, and arises in models of anomalous transport, population dynamics, and finance.  
On bounded domains in $\mathbb{R}^d$, it is typically defined either through the spectral decomposition of the Dirichlet Laplacian (see, e.g., \cite[Section~2.5.1]{lischke}).  
In the setting of metric graphs, we adopt the spectral approach, which is well-suited to combining nonlocal operators with network structures.  

Within this framework, we define the \emph{fractional torsion function} $u_\alpha$ as the unique positive solution of the boundary value problem
\[
(-\Delta_{\mathcal{G}})^\alpha u_\alpha = 1 \quad \text{on }\mathcal{G}, 
\qquad u_\alpha|_{\mathcal{V}_D}=0 ,
\]
where $\mathcal{V}_D\subseteq \mathcal{V}$ is the non-empty set of Dirichlet vertices.
The corresponding \emph{fractional torsional rigidity} is given by
\[
T_\alpha(\mathcal{G}) := \int_{\mathcal{G}} u_\alpha(x)\, dx,
\]
that is, the $L^1$-mass of the torsion function.  
Equivalently, using the spectral decomposition of the Laplacian with eigenpairs $\{(\lambda_n,\varphi_n)\}_{n\geq 1}$, one has
\begin{equation} \label{eq:intro torsion_series}
u_{\alpha} = \big((-\Delta_{\mathcal{G}})^\alpha\big)^{-1} 1 
= \sum_{n=1}^\infty \frac{\langle 1, \varphi_n \rangle}{\lambda_n^\alpha}\, \varphi_n,
\qquad 
T_\alpha(\mathcal{G}) = \sum_{n=1}^\infty \frac{|\langle 1, \varphi_n \rangle|^2}{\lambda_n^\alpha}.
\end{equation}
The functional framework is given by the fractional Sobolev space
\begin{equation}
    H_0^\alpha(\mathcal{G}) := \left\{ u \in L^2(\mathcal{G}) \;\middle|\; 
    \sum_{n=1}^\infty \lambda_n^{\alpha} |\langle u, \varphi_n \rangle|^2 < \infty,
    \ \ u|_{\mathcal{V}_D}=0 \right\},
\end{equation}
endowed with the norm
\[
\|u\|_{H^\alpha_0(\mathcal{G})}^2 := \sum_{n=1}^\infty \lambda_n^\alpha |\langle u, \varphi_n \rangle|^2.
\]
This norm dominates the $L^2$–norm, i.e.\ there exists $c_\alpha>0$ such that
\[
\|u\|_{L^2(\mathcal{G})}^2 \leq c_\alpha^{-1}\, \|u\|_{H_0^\alpha(\mathcal{G})}^2,
\]
but the converse inequality does not hold in general.
For $\alpha > \tfrac12$, functions in $H^\alpha_0(\mathcal{G})$ admit well-defined traces at the vertices, so the condition $u|_{\mathcal{V}D}=0$ can be imposed in the usual Sobolev sense, and in fact $u_\alpha$ is continuous on $\mathcal{G}\setminus\mathcal{V}_D$. When $\alpha \leq \tfrac12$, however, vertex traces are not defined, so the boundary condition cannot be interpreted in this way. Since throughout we adopt the spectral definition of the fractional Laplacian via its eigenfunction expansion, the requirement $u|_{\mathcal{V}_D}=0$ is in fact encoded in the choice of Dirichlet eigenbasis, which makes the formulation consistent for all $\alpha\in(0,1)$. The key distinction is that for $\alpha>\tfrac12$ one can also appeal to a Sobolev–analytic interpretation with better regularity properties, whereas for $\alpha\leq\tfrac12$ the solution is defined spectrally but enjoys weaker regularity.

It is worth noting that classical rearrangement arguments, which play a central role in the domain setting \cite{brasco} and have been adapted to metric graphs in \cite{MugnoloPlumer2023,OzcanTaeufer2024}, is not applied here.  
The reason lies in the change of the underlying function space and the norm structure.

A key tool in our analysis is the following variational characterization of fractional torsional rigidity:
\begin{equation}\label{eqn:variationalcharacterization}
T_{\alpha}(\mathcal{G})
= \sup_{f \in H_0^\alpha(\mathcal{G})} 
\frac{\left( \int_{\mathcal{G}} f(x)\, dx \right)^2}
     {\langle f, (-\Delta_{\mathcal{G}})^\alpha f \rangle}
= \max_{f \in H_0^\alpha(\mathcal{G})} 
\frac{\left( \int_{\mathcal{G}} f(x)\, dx \right)^2}
     {\langle f, (-\Delta_{\mathcal{G}})^\alpha f \rangle}.
\end{equation}
This identity enables us to derive both upper and lower bounds for $T_\alpha(\mathcal{G})$ directly by means of surgery principles, a technique that has proved powerful in the analysis of extremal spectral quantities on graphs. 
In particular, we show that, for graphs of fixed total length, the fractional torsional rigidity is maximized by the interval and minimized by the flower graph  
\begin{equation}\label{eq:intro_bounds}
|\mathcal{E}|\; \frac{8\,L^{1+2\alpha}}{\pi^{2(1+\alpha)}}\Big(1-2^{-2(1+\alpha)}\Big)\,\zeta\big(2(1+\alpha)\big) 
\;\leq\; T_\alpha(\mathcal{G}) \;\leq\; 
\frac{8\,2^{2\alpha}\,|\mathcal{G}|^{2\alpha+1}}{\pi^{2+2\alpha}} \left(1 - 2^{-(2+2\alpha)}\right)\zeta(2+2\alpha),
\end{equation}
where  $\zeta(s)$ denotes the Riemann zeta function,
\[
\zeta(s) := \sum_{n=1}^\infty \frac{1}{n^s}, \qquad \Re(s) > 1.
\]
 Although this conclusion parallels the classical case, the methods required to establish it are substantially different.  Moreover, they are consistent with the classical case $\alpha=1$, in the sense that the bounds in \eqref{eq:intro_bounds} reduce numerically to the known estimates for the classical torsional rigidity.

The paper is organized as follows.  
In Chapter~\ref{section: fractional laplace} we recall the basic framework of metric graphs, the Laplacian, and its fractional powers.  
Chapter~\ref{section: fractional torsional rigidity} introduces the notions of the fractional torsion function and torsional rigidity, and establishes their convergence and positivity for $\alpha \in (0,1)$.  
In Chapter~\ref{section: fundamental examples} we compute explicit examples, including the interval and flower graphs. 
Chapter~\ref{section: surgery principles} develops the functional analytic framework and proves the variational characterization.
In Chapter~\ref{section: variational char} we establish surgery principles for fractional torsional rigidity.  
Finally, Chapter~\ref{section:bounds} derives upper and lower bounds.

\section{Fractional Laplace Operator on Metric Graphs}\label{section: fractional laplace}
We begin by recalling the basic notions of metric graphs and Laplace operators on them; see \cite{BK,mug3} for comprehensive references.
\begin{definition}
A \emph{metric graph} $\mathcal{G}$ is a topological space associated with a finite combinatorial graph 
$\mathsf{G} = (\mathcal{V}, \mathcal{E})$, where to each edge $\mathsf{e} \in \mathcal{E}$ one associates an interval of length 
$\ell_{\mathsf{e}} \in (0,\infty)$, and these intervals are glued together at their endpoints according to the adjacency structure of $\mathsf{G}$.
Every edge $\mathsf{e} \in \mathcal{E}$ can be identified with the interval $[0,\ell_{\mathsf{e}}]$.

\end{definition}

Such structures, sometimes called \emph{topological networks}, come naturally equipped with the shortest-path distance and the one-dimensional Lebesgue measure~\cite{nica,mug3}, and although infinite graphs have also been studied~\cite{KostenkoN-19,KostenkoMN-22,DuefelKennedyMugnoloPluemerTaeufer-22,KennedyMugnoloTaeufer-2024_preprint}, we focus here on finite graphs to avoid analytical complications.

\begin{definition}\label{assump}
	A metric graph $\mathcal{G}$ is called \emph{connected} if every pair of points $x,y \in \mathcal{G}$ can be joined by a continuous path lying in $\mathcal{G}$.  
	The \emph{total length} of $\mathcal{G}$ is defined as
	\[
	|\mathcal{G}| := \sum_{\mathsf{e} \in \mathsf{E}} \ell_{\mathsf{e}} ,
	\]
	while $\operatorname{dist}_{\mathcal{G}}$ denotes the intrinsic shortest-path metric on $\mathcal{G}$.
\end{definition}

When considering functions $u : \mathcal{G} \to \mathbb{C}$ we denote by $u_{\mathsf{e}}$ the restriction of $u$ to the edge $\mathsf{e}$.
For a vertex $\mathsf{v}$ incident with an edge $\mathsf{e}$, we denote by 
	\[
	\frac{\partial u_{\mathsf{e}}}{\partial n}(\mathsf{v})
	\]
	the derivative of $u$ along $\mathsf{e}$ at the endpoint $\mathsf{v}$, taken in the direction pointing into $\mathsf{v}$.
   For $1 \leq p \leq \infty$, the space $L^p(\mathcal{G})$ is the direct sum
\[
L^p(\mathcal{G}) := \bigoplus_{e\in\mathcal{E}} L^p(0,\ell_e),
\]
with norm
\[
\|u\|_{L^p(\mathcal{G})} :=
\begin{cases}
\left( \displaystyle\sum_{e\in\mathcal{E}} \int_0^{\ell_e} |u_e(x)|^p \, dx \right)^{1/p}, & 1 \leq p < \infty, \\[1.2em]
\displaystyle\max_{e\in\mathcal{E}} \, \operatorname*{ess\,sup}_{x \in (0,\ell_e)} |u_e(x)|, & p = \infty .
\end{cases}
\]
We distinguish a subset of vertices $\emptyset \neq \mathcal{V}_D \subseteq \mathcal{V}$ where Dirichlet boundary conditions are imposed. On the remaining vertices $\mathcal{V}\setminus \mathcal{V}_D$, we impose standard \emph{Kirchhoff conditions}, i.e.,\[
f \ \text{is continuous at each vertex } v, \quad \text{ and } \quad
\sum_{e\sim v} \frac{\partial u_{\mathsf{e}}}{\partial n}(\mathsf{v}) = 0,
\]
where the sum runs over all edges incident to $v$, and $\partial_\nu$ denotes the outward derivative along an edge.
The corresponding \emph{Laplacian} $-\Delta_{\mathcal{G}}$ acts as negative of the second derivative on each edge,
\[
(-\Delta_{\mathcal{G}} f)_e = -f_e'' \quad \text{for } e\in\mathcal{E},
\]
with domain consisting of all edgewise $H^2$-functions that satisfy the vertex conditions. The operator $-\Delta_{\mathcal{G}}$ is self-adjoint, lower semibounded and has compact resolvent, hence its spectrum consists of a discrete sequence of eigenvalues
\[
0 < \lambda_1 \leq \lambda_2 \leq \cdots \to \infty,
\]
with corresponding orthonormal eigenfunctions $\{\varphi_n\}_{n=1}^\infty$ (see \cite{dunford}). Any $f \in L^2(\mathcal{G})$ can be expanded as 
\[
f = \sum_{n=1}^\infty \langle f, \varphi_n \rangle \varphi_n,
\]
where $\langle \cdot, \cdot \rangle$ denotes the inner product in $L^2(\mathcal{G})$. \\
For $\alpha \in (0,1)$, the fractional Laplacian $(-\Delta_{\mathcal{G}})^\alpha$ is defined by
\begin{equation}\label{eq:fractional-laplacian}
	(-\Delta_{\mathcal{G}})^\alpha f := \sum_{n=1}^\infty \lambda_n^\alpha \langle f, \varphi_n \rangle \varphi_n, \quad f \in \mathrm{D}((-\Delta_{\mathcal{G}})^\alpha),
\end{equation}
with domain
\[
\mathrm{D}((-\Delta_{\mathcal{G}})^\alpha) := \Bigl\{ f \in L^2(\mathcal{G}) : \sum_{n=1}^\infty \lambda_n^{2\alpha} |\langle f, \varphi_n \rangle|^2 < \infty \Bigr\}.
\]
In particular, the operator $(-\Delta_{\mathcal{G}})^\alpha$ defined above is self-adjoint and positive on $L^2(\mathcal{G})$.

\section{Fractional Torsion Function and Torsional Rigidity}\label{section: fractional torsional rigidity}
In this section, we introduce the fractional torsion function on a compact metric graph and study its basic properties. 
The torsion function is defined as the solution to a Poisson-type problem involving the fractional Laplacian, 
and its $L^1$–norm yields the fractional torsional rigidity of the graph.

\begin{definition}[Fractional Torsion Function]
	Let $\alpha \in (0,1)$. The \emph{fractional torsion function} $u_{\alpha}$ is defined as the unique solution to
	\begin{equation}\label{eq:torsion_eq}
		(-\Delta_{\mathcal{G}})^\alpha u_{\alpha} = 1 \quad \text{on } \mathcal{G}.
	\end{equation}
\end{definition}
	Since the Laplacian $-\Delta_{\mathcal{G}}$ has strictly positive spectrum, the operator $(-\Delta_{\mathcal{G}})^\alpha$ is invertible, and the solution can be expressed via the spectral decomposition as
	\begin{equation}\label{eq:torsion_series}
		u_{\alpha} = ((-\Delta_{\mathcal{G}})^\alpha)^{-1} 1 = \sum_{n=1}^\infty \frac{\langle 1, \varphi_n \rangle}{\lambda_n^\alpha} \varphi_n.
	\end{equation}
 Strictly speaking, $u_\alpha$ also depends on the choice of Dirichlet vertices $\mathcal{V}_D$, but we suppress this dependence in the notation for simplicity. 
In the next lemma, we will show that the fractional torsion function is well-defined via its spectral series, thereby establishing its continuity on compact subsets of the graph outside the Dirichlet vertices.
\begin{lemma}
	The series \eqref{eq:torsion_series} converges uniformly on compact subsets of $\mathcal{G} \setminus \mathcal{V}_D$, and hence defines a continuous function
	\(
u_{\alpha} \in C(\mathcal{G} \setminus \mathcal{V}_D)
	\) if $\alpha > \tfrac{1}{2}$.
\end{lemma}

\begin{proof}
Let $\{\varphi_n\}$ be the $L^{2}$–orthonormal eigenfunctions of the Laplacian on the compact metric graph $\mathcal G$, with eigenvalues $\lambda_n>0$:
\[
-\,\varphi_n''=\lambda_n \varphi_n \quad\text{on each edge},\qquad 
\|\varphi_n\|_{L^{2}(\mathcal G)}=1.
\]
Multiplying the eigenvalue equation by $\varphi_n$ and integrating over the whole graph,
\[
\int_{\mathcal G} |\varphi_n'(x)|^{2}\,dx
  = \lambda_n \int_{\mathcal G} |\varphi_n(x)|^{2}\,dx
  = \lambda_n,
\]
so
\[
\|\varphi_n'\|_{L^{2}(\mathcal G)} = \sqrt{\lambda_n}.
\]
For any edge $e$ of the graph (a finite interval), the one–dimensional Sobolev inequality gives (see, e.g. \cite[§5.6]{Evans})
\[
\|u\|_{L^{\infty}(e)}
 \le C_{e}\bigl(\|u\|_{L^{2}(e)} + \|u'\|_{L^{2}(e)}\bigr)
 \qquad \forall\, u \in H^{1}(e).
\]
Because $\mathcal G$ has finitely many edges of bounded length, there exists a constant
$C>0$ independent of $e$ such that
\[
\|u\|_{L^{\infty}(\mathcal G)}
 \le C\bigl(\|u\|_{L^{2}(\mathcal G)} + \|u'\|_{L^{2}(\mathcal G)}\bigr).
\]
Applying this to $u=\varphi_n$ and using $\|\varphi_n\|_{L^{2}}=1$ and
$\|\varphi_n'\|_{L^{2}}=\sqrt{\lambda_n}$ yields
\[
\|\varphi_n\|_{L^{\infty}(\mathcal G)}
 \le C\bigl(1+\sqrt{\lambda_n}\bigr)
 \le C'\sqrt{\lambda_n},
\]
where $C'$ absorbs the case of small $\lambda_n$.
Since $|\langle 1,\varphi_n\rangle| \le |\mathcal G|^{1/2}$,
\[
\left| \frac{\langle 1,\varphi_n\rangle}{\lambda_n^\alpha}\,\varphi_n(x) \right|
   \le K\,\lambda_n^{\tfrac12-\alpha}.
\]
Weyl’s law for compact metric graphs gives $\lambda_n\sim n^{2}$, so
$\sum_{n} \lambda_n^{1/2-\alpha}$ converges whenever $\alpha>\tfrac12$.
By the Weierstrass $M$–test the series
\[
\sum_{n=1}^{\infty} \frac{\langle 1,\varphi_n\rangle}{\lambda_n^\alpha}\,\varphi_n(x)
\]
converges uniformly on every compact subset of $\mathcal G\setminus\mathcal V_D$,
and therefore defines a continuous function
\(u_\alpha \in C(\mathcal G\setminus\mathcal V_D)\).
\end{proof}
Subsequently, we will prove that the fractional torsion function is strictly positive almost everywhere on the interior of the graph, highlighting its role as a positivity-improving solution to the Poisson-type problem.
	\begin{theorem}
Let $\alpha \in (0,1)$.		
The fractional torsion function $u_{\alpha}$ satisfies
		\[
		u_{\alpha}(x) > 0 \quad \text{for almost every } x \in \mathcal{G} \setminus \mathcal{V}_D.
		\]
	\end{theorem}
	
\begin{proof}
Consider the heat semigroup $(e^{t\Delta_{\mathcal{G}}})_{t\geq 0}$ on $L^2(\mathcal{G})$ with Dirichlet conditions. Its generator is $\Delta_{\mathcal{G}}$, which is self-adjoint and has spectrum $\{-\lambda_n\}_{n=1}^\infty$ where $\lambda_n > 0$ are the eigenvalues of $-\Delta_{\mathcal{G}}$.
Let $u \equiv 1$ be the constant function. Since $\mathcal{G}$ has finite measure, $u$ is a quasi-interior point of $L^2(\mathcal{G})_+$. The spectral bound is $ -\lambda_1$, and the associated spectral projection $P$ satisfies
\[
Pu = \langle 1, \varphi_1\rangle \varphi_1 \neq 0,
\]
because $\varphi_1 > 0$ on $\mathcal{G}\setminus\mathcal{V}_D$ and $1 > 0$.
By Theorem~10.2.1 of \cite{GluckThesis}, the semigroup $(e^{t\Delta_{\mathcal{G}}})_{t\geq 0}$ is individually eventually strongly positive with respect to $u=1$. That is, there exists $t_0 > 0$ such that for all $t > t_0$,
\[
e^{t\Delta_{\mathcal{G}}}1 \geq \varepsilon \quad \text{almost everywhere on } \mathcal{G}\setminus\mathcal{V}_D
\]
for some $\varepsilon > 0$.
Now using the integral representation for the fractional Laplacian,
\[
u_\alpha = (-\Delta_{\mathcal{G}})^{-\alpha}1 = \frac{1}{\Gamma(\alpha)} \int_0^\infty t^{\alpha-1} e^{t\Delta_{\mathcal{G}}}1 \, dt.
\]
Split the integral:
\[
u_\alpha = \frac{1}{\Gamma(\alpha)} \left( \int_0^{t_0} t^{\alpha-1} e^{t\Delta_{\mathcal{G}}}1 \, dt + \int_{t_0}^\infty t^{\alpha-1} e^{t\Delta_{\mathcal{G}}}1 \, dt \right).
\]
The first term is nonnegative since $e^{t\Delta_{\mathcal{G}}}$ is positivity preserving. For the second term we have
\[
\int_{t_0}^\infty t^{\alpha-1} e^{t\Delta_{\mathcal{G}}}1 \, dt \;\geq\; \varepsilon \int_{t_0}^\infty t^{\alpha-1} dt \cdot 1 \;=\; C_\varepsilon \, 1,
\]
where $C_\varepsilon > 0$. Therefore,
\[
u_\alpha \;\geq\; \frac{C_\varepsilon}{\Gamma(\alpha)} > 0
\]
almost everywhere on $\mathcal{G}\setminus\mathcal{V}_D$.
\end{proof}

		\subsection*{Fractional Torsional Rigidity}
		
		Having established the existence, uniqueness, and positivity of the fractional torsion function, we now introduce the concept of \emph{fractional torsional rigidity} of a metric graph. Intuitively, the torsional rigidity measures the “total flexibility” of the graph under a unit load and is defined as the total mass of the torsion function.
		
		\begin{definition}
			Let $u$ be the fractional torsion function on $\mathcal{G}$. The \emph{fractional torsional rigidity} is defined by
			\[
			T_{\alpha}(\mathcal{G}) := \int_{\mathcal{G}} u_{\alpha}(x) \, dx.
			\]
		\end{definition}
		Using the spectral representation \eqref{eq:torsion_series} for $u$, we immediately obtain
		\begin{align}\label{eq:torsional_rigidity_series}
			T_{\alpha}(\mathcal{G}) 
			&= \int_{\mathcal{G}} \sum_{n=1}^\infty \frac{\langle 1, \varphi_n \rangle}{\lambda_n^\alpha} \varphi_n(x) \, dx \nonumber\\
			&= \sum_{n=1}^\infty \frac{\langle 1, \varphi_n \rangle}{\lambda_n^\alpha} \int_{\mathcal{G}} \varphi_n(x) \, dx \nonumber \\
			&= \sum_{n=1}^\infty \frac{|\langle 1, \varphi_n \rangle|^2}{\lambda_n^\alpha}.
		\end{align}
		This expression provides the \emph{spectral representation} of the fractional torsional rigidity. As a simple consequence, one obtains the following variational lower bound
\begin{equation}\label{eq:variational_lower_bound}
		T_{\alpha}(\mathcal{G}) \ge \frac{|\langle 1, \varphi_1 \rangle|^2}{\lambda_1^\alpha}.
\end{equation}

\begin{lemma}
Let $\mathcal{G}$ be a compact metric graph. Then the series \eqref{eq:torsional_rigidity_series}
converges for every $\alpha>0$, hence $T_\alpha(\mathcal G)<\infty$.
Moreover,
\[
 T_\alpha(\mathcal G)
\;\le\; \frac{|\mathcal G|}{\lambda_1^{\alpha}}.
\]
\end{lemma}

\begin{proof}
By Bessel’s inequality,
$\sum_{n=1}^\infty |\langle 1,\varphi_n\rangle|^2 = \|1\|_{L^2(\mathcal G)}^2 = |\mathcal G|$.
Since $\lambda_n\ge \lambda_1>0$ for all $n$,
\[
\sum_{n=1}^\infty \frac{|\langle 1,\varphi_n\rangle|^2}{\lambda_n^\alpha}
\le
\frac{1}{\lambda_1^\alpha}\sum_{n=1}^\infty |\langle 1,\varphi_n\rangle|^2
= \frac{|\mathcal G|}{\lambda_1^\alpha}<\infty. \qedhere
\]
\end{proof}
\section{Fundamental Examples}\label{section: fundamental examples}

In this section, we illustrate the theory developed above by computing the fractional torsion function and the corresponding fractional torsional rigidity for several fundamental graphs. These examples include the interval and flower graphs, all of which allow for explicit spectral calculations. For each case, the fractional torsion function is obtained via its spectral representation, and the fractional torsional rigidity is computed by integrating the torsion function. Notably, in all examples, the formulas are consistent with the classical case $\alpha = 1$, recovering the standard torsion function and torsional rigidity. These computations demonstrate both the effectiveness of the spectral approach and the influence of the fractional exponent $\alpha$ on the torsional properties of the graph.

We start with the simplest case, the interval, which allows for fully explicit calculations.
\begin{example}\label{ex: interval}
	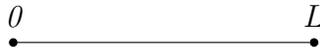
\begin{figure}\label{interval graph}
		\centering
		\begin{tikzpicture}
		\node[draw, circle, inner sep=1pt, fill, label=above:0] (A) at (0, 0) {};
		\node[draw, circle, inner sep=1pt, fill, label=above:$L$] (B) at (4, 0) {};
		\draw[ ] (A) -- (B);
	\end{tikzpicture}
	\caption{The interval graph of length $L$, serving as the simplest one-dimensional domain where the fractional torsion function and torsional rigidity can be computed explicitly.}
	\end{figure}

	Consider the interval graph (see Figure \ref{interval graph}) $\mathcal{G} = [0,L]$ with boundary conditions
	\[
	u_\alpha(0) = 0, \quad u_\alpha'(L) = 0.
	\]
	The eigenvalues and orthonormal eigenfunctions of $-\Delta_{[0,L]}$ under these boundary conditions are
	\[
	\lambda_n = \left(\frac{(2n-1)\pi}{2L}\right)^2, \quad
	\varphi_n(x) = \sqrt{\frac{2}{L}} \sin\left(\frac{(2n-1)\pi x}{2L}\right), \quad n \ge 1.
	\]
Using the spectral decomposition \eqref{eq:torsion_series}of the torsion function 
	with Fourier coefficients
	\[
	\langle 1, \varphi_n \rangle = \int_0^L \sqrt{\frac{2}{L}} \sin\left(\frac{(2n-1)\pi x}{2L}\right) dx
	= \frac{2\sqrt{2L}}{(2n-1)\pi},
	\]
 we obtain the explicit series representation of the torsion function
	\begin{equation}\label{torsion function of interval graph}
	    u_\alpha(x) = \sum_{n=1}^\infty \frac{4 (2L)^{2\alpha}}{(2n-1)^{1+2\alpha} \pi^{1+2\alpha}} \sin\left(\frac{(2n-1)\pi x}{2L}\right).
	\end{equation}
	The corresponding fractional torsional rigidity is
	\[
	T_\alpha(\mathcal{G},\{0\}) = \int_0^L u_\alpha(x)\, dx
	= \sum_{n=1}^\infty \frac{8 \cdot 2^{2\alpha} L^{2\alpha+1}}{(2n-1)^{2+2\alpha} \pi^{2+2\alpha}}.
	\]
	Using the standard identity for the Riemann zeta function,
	\begin{equation}\label{eq: Riemann zeta}
	    \sum_{n=1}^\infty \frac{1}{(2n-1)^s} = \left(1 - 2^{-s}\right) \zeta(s),\quad \zeta(s) := \sum_{n=1}^\infty \frac{1}{n^s}, \qquad \Re(s) > 1.
	\end{equation}
	we can write
	\begin{equation}\label{torsional rigidity of interval graph}
	    	T_\alpha(\mathcal{G},\{0\}) = \frac{8 \cdot 2^{2\alpha} L^{2\alpha+1}}{\pi^{2+2\alpha}} \left(1 - 2^{-(2+2\alpha)}\right) \zeta(2+2\alpha).
	\end{equation}
	In the classical case $\alpha = 1$, this reduces to
	\[
	T_1(\mathcal{G},\{0\}) = \frac{L^3}{3},
	\]
	which coincides with the well-known torsional rigidity for an interval with Dirichlet–Neumann boundary conditions. This illustrates the consistency of the fractional definition with the classical Laplacian (see \cite{MugnoloPlumer2023}).
\end{example}

\begin{example}\label{ex: flower}
	\begin{figure}[h] \label{fig: flower}
		\centering
			\begin{tikzpicture}
				\coordinate (v1) at (60*1:2);
				\coordinate (v2) at (60*2:2);
				\coordinate (v3) at (60*3:2);
				\coordinate (v4) at (60*4:2);
				\coordinate (v5) at (60*5:2);
				\coordinate (v6) at (60*6:2);
				\foreach \i in {1,2,3,4,5,6} {
					\draw[line width=.7pt] (0,0) to [out=60*\i-30,in=60*\i+270] (v\i);
					\draw[line width=.7pt] (0,0) to [out=60*\i+30,in=60*\i+90] (v\i);
				}
				\draw[fill=black, line width=1pt] (0,0) circle (2pt);
			\end{tikzpicture}
	\caption{The flower graph, consisting of several loops attached to a common vertex. This structure highlights how cycles influence the behavior of the fractional torsion function.}
	\end{figure}
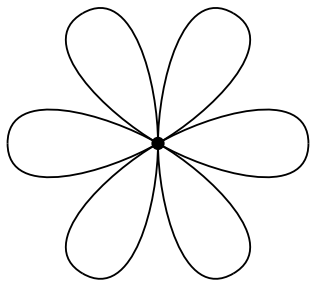
	Let $\mathcal{F}$ be an equilateral flower graph (see Figure \ref{fig: flower}) with $N$ edges of length $L>0$:
	\[
	\mathcal{F} = \bigcup_{j=1}^N e_j,
	\]
	with Dirichlet boundary conditions at all vertices:
	\[
	u = 0 \quad \text{at all vertices of } \mathcal{F}.
	\]
	Each edge can be treated as an interval $[0,L]$ with Dirichlet conditions at both ends. The  eigenpairs of the Laplacian $-\Delta_{\mathcal{V}}$ are
	\[
	\lambda_n = \left(\frac{n\pi}{L}\right)^2, \quad \varphi_n(x) = \sqrt{\frac{2}{L}} \sin\left(\frac{n\pi x}{L}\right), \quad n\ge 1,
	\]
	and on $\mathcal{F}$ the eigenfunctions supported on $e_j$ are
	\[
	\varphi_{j,n}(x) =
	\begin{cases} 
		\sqrt{\frac{2}{L}}\sin\left(\frac{n\pi x}{L}\right), & x\in e_j,\\
		0, & \text{otherwise}.
	\end{cases}
	\]
	Computing the inner product
	\[
	\langle \varphi_{j,n},1\rangle =
	\begin{cases} 
		\dfrac{2\sqrt{2L}}{n\pi}, & n \text{ odd}, \\[1mm]
		0, & n \text{ even}.
	\end{cases}
	\]
    Using the spectral decomposition \eqref{eq:torsion_series}, the torsion function on each edge is \[ u_{\alpha}(x)=\sum_{n \text{ odd }} \frac{4L^{2\alpha}}{(n\pi)^{1+2\alpha}}\sin\Big( \frac{n\pi x}{L} \Big).\]
	Similarly, the torsional rigidity is computed by using \eqref{eq:torsional_rigidity_series} and the series identity for the Riemann zeta function \eqref{eq: Riemann zeta} 
	\[
	T_\alpha(\mathcal{F},\mathcal{V}) = N \cdot 8 L^{1+2\alpha} \pi^{-2(1+\alpha)} \left(1-2^{-2(1+\alpha)}\right) \zeta\big(2(1+\alpha)\big),
	\]
	and for $\alpha=1$:
	\[
	T_1(\mathcal{F},\mathcal{V})  = \frac{N L^3}{12} = \frac{|\mathcal{G}|^3}{12 N^2},
	\]
    which is the torsional rigidity of the flower graph in \cite{MugnoloPlumer2023}.
\end{example}

\section{Variational Characterization} \label{section: variational char}
\noindent
A central feature of fractional Laplacians on metric graphs is their close connection to variational principles. 
In analogy with the classical case $\alpha=1$, where torsional rigidity can be expressed as the maximum of a quadratic functional (see \cite{brasco, MugnoloPlumer2023, OzcanTaeufer2024}),  
the fractional torsional rigidity admits a similar characterization in terms of the energy associated with the space $H^\alpha_0(\mathcal{G})$. 
This perspective is particularly useful, since it not only establishes existence and uniqueness of the torsion function, 
but also provides a natural framework for deriving estimates and comparison principles. 
In what follows, we introduce the relevant fractional Sobolev spaces, show that they are well-defined Hilbert spaces on $\mathcal{G}$, 
and then prove that torsional rigidity can be described as the supremum of a suitable concave functional. 

We define the fractional Sobolev space
\begin{equation}\label{set:sobolevspace}
    H_0^\alpha(\mathcal{G}) := \left\{ u \in L^2(\mathcal{G}) \;\middle|\; \sum_{n=1}^\infty \lambda_n^{\alpha} |\langle u, \varphi_n \rangle|^2 < \infty \text{ and } u \text{ vanishes on } \mathcal{V}_D \right\},
\end{equation}
where $\{\lambda_n\}_{n=1}^\infty$ and $\{\varphi_n\}_{n=1}^\infty$ denote the eigenvalues and $L^2$-orthonormal eigenfunctions of the Laplacian on $\mathcal{G}$ with Dirichlet vertex conditions on $\mathcal{V}_D$, Neumann elsewhere.
For any \(\alpha \in (0,1) \), define the fractional Sobolev norm as
\[
\|u\|_{H^\alpha_0(\mathcal{G})}^2 := \sum_{n=1}^\infty \lambda_n^\alpha |\langle u, \varphi_n \rangle|^2.
\]
\begin{remark}
The above definition of $H_0^\alpha(\mathcal{G})$ is understood in the spectral sense,
that is, through the eigenfunction expansion of the Laplacian with Dirichlet conditions
on $\mathcal{V}_D$. For $\alpha>\tfrac12$, this coincides with the usual Sobolev
space interpretation, since vertex traces are well defined and the condition
$u|_{\mathcal{V}_D}=0$ can be imposed in the standard way. When $\alpha\le\tfrac12$,
however, traces are not defined, and the spectral formulation provides a natural
extension that remains consistent for all $\alpha\in(0,1)$.
\end{remark}
Before presenting the variational characterization of torsional rigidity, 
we first establish that the fractional Sobolev space $H_0^\alpha(\mathcal{G})$ is a well-defined Hilbert space. 
In particular, the following proposition shows that the spectral definition of the fractional Sobolev norm indeed satisfies all the properties of a norm 
and provides a useful bound in terms of the $L^2$-norm.
\begin{proposition}\label{prop:Halpha_norm}
Let $\alpha \in (0,1)$ and let $H_0^\alpha(\mathcal{G})$ be defined as in \eqref{set:sobolevspace}.
The map
\[
\|u\|_{H_0^\alpha(\mathcal{G})} := \Biggl( \sum_{n=1}^\infty \lambda_n^\alpha |\langle u, \varphi_n \rangle|^2 \Biggr)^{1/2}
\]
defines a norm on $H_0^\alpha(\mathcal{G})$. Moreover, it dominates the $L^2$-norm in the sense that
\begin{equation}\label{eqn:normrelation_proposition}
\|u\|_{L^2(\mathcal{G})} \leq \lambda_1^{-\alpha} \|u\|_{H_0^\alpha(\mathcal{G})}.
\end{equation}
\end{proposition}

\begin{proof}
Positivity and definiteness follow immediately from the fact that $\lambda_n^\alpha > 0$ and the orthonormality of $\{\varphi_n\}$.  
For homogeneity, let $\lambda \in \mathbb{R}$; then
\[
\|\lambda u\|_{H_0^\alpha}^2 = \sum_{n=1}^\infty \lambda_n^\alpha |\langle \lambda u, \varphi_n \rangle|^2 = |\lambda|^2 \sum_{n=1}^\infty \lambda_n^\alpha |\langle u, \varphi_n \rangle|^2 = |\lambda|^2 \|u\|_{H_0^\alpha}^2.
\]
For the triangle inequality, consider the linear map
\[
u \mapsto \bigl\{ \lambda_n^{\alpha/2} \langle u, \varphi_n \rangle \bigr\}_{n=1}^\infty \in \ell^2.
\]
The $\ell^2$ norm satisfies the triangle inequality, which immediately implies
\[
\|u+v\|_{H_0^\alpha} \leq \|u\|_{H_0^\alpha} + \|v\|_{H_0^\alpha}, \quad \forall u,v \in H_0^\alpha(\mathcal{G}).
\]
Finally, the equivalence with the $L^2$-norm follows from Parseval's identity:
\[
\|u\|_{L^2(\mathcal{G})}^2 = \sum_{n=1}^\infty |\langle u, \varphi_n \rangle|^2 \leq \lambda_1^{-\alpha} \sum_{n=1}^\infty \lambda_n^\alpha |\langle u, \varphi_n \rangle|^2 = \lambda_1^{-\alpha} \|u\|_{H_0^\alpha}^2.
\]
This completes the proof.
\end{proof}
\begin{remark}\label{rem:norm_operator}
For any $v \in H_0^\alpha(\mathcal{G})$, the homogeneous Sobolev norm can be expressed in terms of the fractional Laplacian:
\begin{equation}\label{relationshipbetweeninnerandnorm}
  \|v\|_{H_0^\alpha(\mathcal{G})}^2 = \sum_{n=1}^\infty \lambda_n^\alpha |\langle v, \varphi_n \rangle|^2 = \langle v, (-\Delta_{\mathcal{G}})^\alpha v \rangle.  
\end{equation}
This follows directly from the spectral definition of $(-\Delta_{\mathcal{G}})^\alpha$, since $v$ can be expanded in the eigenbasis $\{\varphi_n\}$ of the Laplacian. 
\end{remark}
With the norm on $H_0^\alpha(\mathcal{G})$ well-defined and its connection to the fractional Laplacian clarified (see Remark~\ref{rem:norm_operator}), 
we are now ready to formulate the torsional rigidity in a variational framework. 
This perspective not only guarantees the existence and uniqueness of the torsion function, 
but also provides a convenient functional setting to derive estimates and extremal properties.
\begin{theorem}
	Let $\alpha \in (0,1) $ and let $u_\alpha \in H_0^\alpha(\mathcal{G})$ be the fractional torsion function.
	Then the fractional torsional rigidity
	satisfies the variational identity
	\begin{equation}\label{eqn:variationalcharacterization}
	T_{\alpha}(\mathcal{G})= \sup_{f \in H_0^\alpha(\mathcal{G})}\frac{\left( \int_{\mathcal{G}} f(x)\, dx\right) ^2}{\langle f, (-\Delta_{\mathcal{G}})^\alpha f  \rangle }=\max_{f \in H_0^\alpha(\mathcal{G})}\frac{\left( \int_{\mathcal{G}} f(x)\, dx\right) ^2}{\langle f, (-\Delta_{\mathcal{G}})^\alpha f  \rangle }.
	\end{equation}
	Moreover, the supremum is attained at constant factors of $u_\alpha$.
\end{theorem}
\begin{proof}
We introduce the functional
\[
\mathcal{J}(f) := 2 \int_{\mathcal{G}} f(x)\,dx - \langle f, (-\Delta_{\mathcal{G}})^\alpha f \rangle,
\qquad f \in H_0^\alpha(\mathcal{G}),
\]
and show that its unique maximizer is the torsion function \(u_\alpha\), i.e.\ the weak solution of
\[
(-\Delta_{\mathcal{G}})^\alpha u_\alpha = 1 \quad \text{in } H_0^\alpha(\mathcal{G}).
\]
Let \(f \in H_0^\alpha(\mathcal{G})\) and consider perturbations of the form \(f+\varepsilon h\), with \(h \in H_0^\alpha(\mathcal{G})\) and \(\varepsilon \in \mathbb{R}\). Expanding,
\[
\mathcal{J}(f+\varepsilon h) 
= 2\int_{\mathcal{G}} (f+\varepsilon h)\,dx 
- \langle f+\varepsilon h, (-\Delta_{\mathcal{G}})^\alpha (f+\varepsilon h) \rangle.
\]
Differentiating at \(\varepsilon=0\), we obtain the first variation:
\[
\left.\frac{d}{d\varepsilon}\mathcal{J}(f+\varepsilon h)\right|_{\varepsilon=0}
= 2\int_{\mathcal{G}} h(x)\,dx - 2\langle h, (-\Delta_{\mathcal{G}})^\alpha f \rangle.
\]
This expression vanishes for all \(h \in H_0^\alpha(\mathcal{G})\) precisely when
\[
\int_{\mathcal{G}} h(x)\,dx = \langle h, (-\Delta_{\mathcal{G}})^\alpha f \rangle 
\qquad \forall h \in H_0^\alpha(\mathcal{G}),
\]
which, by the Riesz representation theorem, implies
\[
(-\Delta_{\mathcal{G}})^\alpha f = 1.
\]
Hence the critical point of \(\mathcal{J}\) is the torsion function \(u_\alpha\).
Using the identity
\[
\langle v, (-\Delta_{\mathcal{G}})^\alpha v \rangle = \|v\|_{H_0^\alpha(\mathcal{G})}^2,
\]
we may rewrite
\[
\mathcal{J}(v) = 2 \int_{\mathcal{G}} v(x)\,dx - \|v\|_{H_0^\alpha(\mathcal{G})}^2.
\]
The functional is thus the sum of a linear term and a strictly concave quadratic term. Strict convexity of the norm ensures that \(\mathcal{J}\) is strictly concave on \(H_0^\alpha(\mathcal{G})\). 
By Cauchy-Schwarz inequality and equation \eqref{relationshipbetweeninnerandnorm}, we have 
\begin{align*}
\mathcal{J}(v)&=2 \int_{\mathcal{G}} v(x)\, dx - \|v\|_{ {H}^\alpha_0(\mathcal{G})}^2\\ & \leq \|v\|_{L^2(\mathcal{G})}\sqrt{|\mathcal{G}|}-\|v\|_{ {H}^\alpha_0(\mathcal{G})}^2\\ &\leq\sqrt{|\mathcal{G}|} \lambda_1^{-\alpha} \|v\|_{ {H}^\alpha_0(\mathcal{G})}-\|v\|_{ {H}^\alpha_0(\mathcal{G})}^2.
\end{align*}
As \(\|v\|_{H_0^\alpha} \to \infty\), the quadratic term dominates, yielding \(\mathcal{J}(v) \to -\infty\). This property, known as anti-coercivity, guarantees that the supremum of \(\mathcal{J}\) is attained at a unique point, which we already identified as \(u_\alpha\).
For the maximizer \(u_\alpha\),
\[
\mathcal{J}(u_\alpha) 
= 2 \int_{\mathcal{G}} u_\alpha(x)\,dx 
- \langle u_\alpha, (-\Delta_{\mathcal{G}})^\alpha u_\alpha \rangle.
\]
Since \( (-\Delta_{\mathcal{G}})^\alpha u_\alpha = 1\), the second term equals \(\int_{\mathcal{G}} u_\alpha(x)\,dx\). Hence
\[
\mathcal{J}(u_\alpha) = \int_{\mathcal{G}} u_\alpha(x)\,dx = T_{\alpha}(\mathcal{G}).
\]
We verify maximality using Young’s inequality: for \(A,B>0\),
\[
At - \frac{B}{2}t^2 \leq \frac{A^2}{2B}, 
\qquad \text{with equality iff } t = \frac{A}{B}.
\]
For any $u \in H_0^\alpha (\mathcal{G})$, applying this with \(A = 2\int_{\mathcal{G}} u\) and \(B = \langle u, (-\Delta_{\mathcal{G}})^\alpha u \rangle\), we obtain
\[
\mathcal{J}(u)
\leq \max_{\lambda \in \mathbb{R}} \mathcal{J}(\lambda u)
= \frac{\big(\int_{\mathcal{G}} u(x)\,dx\big)^2}
{\langle u, (-\Delta_{\mathcal{G}})^\alpha u \rangle},
\]
with equality at \(\lambda=\frac{2\int_{\mathcal{G}} udx}{\langle u, (-\Delta_{\mathcal{G}})^\alpha u \rangle}\). Since multiplying the argument $u$ by non-zero scalars does not change the quotient invariant under, any multiple of $u_\alpha$ maximizes it.
\end{proof}

\section{Surgery Principles}\label{section: surgery principles}
\noindent
The preceding sections established the definition and variational characterization of the fractional torsional rigidity on metric graphs. 
We now turn to a set of principles describing how the torsional rigidity behaves under simple modifications of the graph structure, such as duplicating edges, unfolding and cutting  cycles, lengthening edges, or gluing vertices. 
These results, often referred to as \emph{surgery principles}, allow us to compare the torsional rigidity of complex graphs to simpler ones, providing both qualitative and quantitative insights. 
In particular, they highlight how the topology and geometry of the graph influence the fractional torsional rigidity.

\noindent \textbf{Doubling edges.} The principle of doubling edges consists in replacing a single edge by two parallel edges of the same length between the same pair of vertices (see Figure \ref{fig:doublingedges}). 
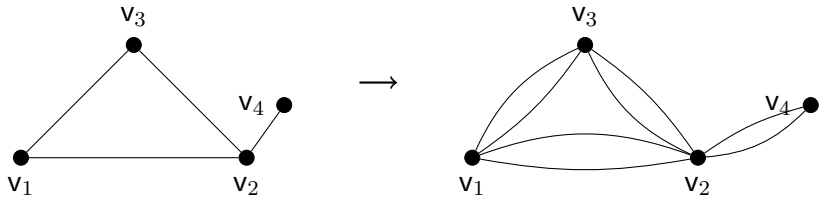
\begin{figure}
    \centering
    \begin{tikzpicture}
        \node[draw, circle, inner sep=2pt, fill,label=below:$\mathsf{v}_1$] (v1) at (0,0) {};
        \node[draw, circle, inner sep=2pt, fill, label=below:$\mathsf{v}_2$] (v2) at (3,0) {};
        \node[draw, circle, inner sep=2pt, fill,label=above:$\mathsf{v}_3$] (v3) at (1.5,1.5) {};
        \node[draw, circle, inner sep=2pt, fill,label=left:$\mathsf{v}_4$] (v4) at (3.5,0.7) {};
        
        \draw (v1) -- (v2);
        \draw (v1) -- (v3);
        \draw (v2) -- (v3);
        \draw (v4) -- (v2);

        \draw[thick,->] (4.5,1) -- (5,1);


    \node[draw, circle, inner sep=2pt, fill,label=below:$\mathsf{v}_1$] (v1) at (6,0) {};
    \node[draw, circle, inner sep=2pt, fill,label=below:$\mathsf{v}_2$] (v2) at (9,0) {};
    \node[draw, circle, inner sep=2pt, fill,label=above:$\mathsf{v}_3$] (v3) at (7.5,1.5) {};
    \node[draw, circle, inner sep=2pt, fill,label=left:$\mathsf{v}_4$] (v4) at (10.5,0.7) {};
    
     \draw[bend right=10] (v1) to (v2);
    \draw[bend left=20] (v1) to (v2);

    \draw[bend right=10] (v1) to (v3);
    \draw[bend left=20] (v1) to (v3);

   \draw[bend right=10] (v2) to (v3);
    \draw[bend left=20] (v2) to (v3);

     \draw[bend right=10] (v4) to (v2);
    \draw[bend left=20] (v4) to (v2);
    \end{tikzpicture}
      \caption{The graph obtained by doubling each edge: for every original connection between two vertices, a second parallel edge is added.}
      \label{fig:doublingedges}
\end{figure}
\begin{theorem}[Edge Doubling]\label{princ:edge-doubling}
	Let $\alpha \in (0,1)$ and  \(\mathcal{G}\) be a compact connected metric graph, and define \(\tilde{\mathcal{G}}\) by replacing each edge with two parallel copies of equal length. Then
	\[
	T_{\alpha}(\mathcal{G})\leq \frac{1}{2} T_{\alpha}(\tilde{\mathcal{G}}).
	\]
\end{theorem}

\begin{proof}
	Given \(f \in H_0^\alpha(\mathcal{G})\), define \(\tilde{f} \in H_0^\alpha(\tilde{\mathcal{G}})\) by copying \(f\) onto each duplicated edge. Then:
	\[
	\langle \tilde{f}, 1 \rangle = 2 \langle f, 1 \rangle, \quad \langle \tilde{f}, (-\Delta_{\tilde{\mathcal{G}}})^{\alpha} \tilde{f} \rangle = 4 \langle f, (-\Delta_{\mathcal{G}})^{\alpha} f \rangle,
	\]
	where the last equality follows from
    \begin{align*}
        \langle \tilde{f}, (-\Delta_{\tilde{\mathcal{G}}})^\alpha \tilde{f} \rangle &=\sum_{n=1}^\infty \tilde{\lambda}_n^\alpha |\langle \tilde{f},\tilde{\phi_n} \rangle |^2\\
        &= 4\sum_{n=1}^\infty \lambda_n^\alpha |\langle f, \phi_n \rangle |^2,
    \end{align*}
    as the eigenvalues of $(-\Delta_{\tilde{\mathcal{G}}})^\alpha$ and $(-\Delta_{\mathcal{G}})^\alpha$ are the same and the corresponding eigenfunctions have the following relationship $ \langle \tilde{f},\tilde{\phi_n} \rangle =2 \langle f, \phi_n \rangle$.
    The variational formula becomes:
	\[
	\frac{\left( \int_{\mathcal{G}} \tilde{f}(x)\, dx\right) ^2}{\langle \tilde{f}, (-\Delta_{\tilde{\mathcal{G}}})^\alpha \tilde{f}  \rangle } = 2\frac{\left( \int_{\mathcal{G}} f(x)\, dx\right) ^2}{\langle f, (-\Delta_{\mathcal{G}})^\alpha f  \rangle },
	\]
	yielding the inequality.
\end{proof}
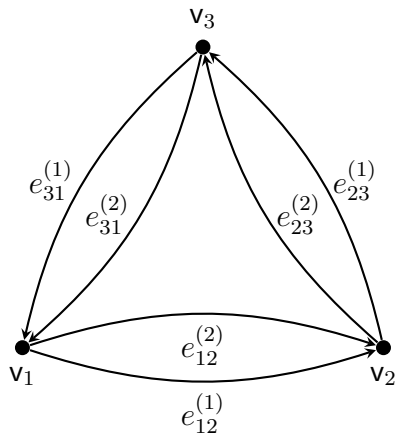
\begin{figure}
\centering
\begin{tikzpicture}[>=stealth,scale=1.6]
  \node[fill,circle,inner sep=2pt,label=below:$\mathsf{v}_1$] (v1) at (0,0) {};
  \node[fill,circle,inner sep=2pt,label=below:$\mathsf{v}_2$] (v2) at (3,0) {};
  \node[fill,circle,inner sep=2pt,label=above:$\mathsf{v}_3$]  (v3) at (1.5,2.5) {};

  \foreach \a/\b in {v1/v2,v2/v3,v3/v1}{
    \draw[->, ultra thin, gray!30, bend left=18]  (\a) to (\b);
    \draw[->, ultra thin, gray!30, bend right=18] (\a) to (\b);
  }

  \draw[->, thick, bend right=18]
      (v1) to node[below,pos=0.5,black] {$e_{12}^{(1)}$} (v2);
  \draw[->,  thick, bend right=18]
      (v2) to node[right,pos=0.5,black] {$e_{23}^{(1)}$} (v3);
  \draw[->,  thick, bend right=18]
      (v3) to node[left,pos=0.5,black]  {$e_{31}^{(1)}$} (v1);
  \draw[->,   thick,   bend left=18]
      (v1) to node[below,pos=0.5,black] {$e_{12}^{(2)}$} (v2);
  \draw[->,   thick,   bend left=18]
      (v2) to node[right,pos=0.5,black] {$e_{23}^{(2)}$} (v3);
  \draw[->, thick,   bend left=18]
      (v3) to node[left,pos=0.5,black]  {$e_{31}^{(2)}$} (v1);
\end{tikzpicture}
\caption{Triangle with each edge doubled. Label the parallel edges $e_{ij}^{(1)},e_{ij}^{(2)}$. The Eulerian closed trail shown in red traverses the edges in order $e_{12}^{(1)},e_{23}^{(1)},e_{31}^{(1)},e_{12}^{(2)},e_{23}^{(2)},e_{31}^{(2)}$, i.e.\ $\mathsf{v}_1\to\mathsf{v}_2\to\mathsf{v}_3\to\mathsf{v}_1\to\mathsf{v}_2\to\mathsf{v}_3\to\mathsf{v}_1$.}
\label{fig:triangle-doubled}
\end{figure}
\begin{theorem}[Unfolding to a Cycle]\label{princ:unfold-cycle}
	Let $\alpha \in (0,1)$ and $\mathcal{G}$ be a compact, connected metric graph in which all vertex degrees are even. Then, there exists a cycle graph $\mathcal{C}$ of the same total length as $\mathcal{G}$ such that
	\[
	T_{\alpha}(\mathcal{G}) \leq T_{\alpha}(\mathcal{C}).
	\]
\end{theorem}

\begin{proof}
Since every vertex in \(\mathcal{G}\) has even degree, there exists a closed Eulerian circuit \(\gamma\) traversing each edge exactly once (see \cite{Karreskog2015} for more detailed information about the Eulerian cycle technique; for an illustration, see Figure~\ref{fig:triangle-doubled}). Let \(\mathcal{C}\) be a cycle graph of total length \(|\mathcal{G}|\), parameterized by \(\gamma\), and define the map \(\Phi: \mathcal{C} \to \mathcal{G}\) following \(\gamma\). This map is surjective and measure-preserving (up to vertices). For any \(v \in H_0^\alpha(\mathcal{G})\), define its unfolding \(\tilde v = v \circ \Phi\) on \(\mathcal{C}\). Then
\[
\int_{\mathcal{C}} \tilde v = \int_{\mathcal{G}} v, \qquad 
\|\tilde v\|_{L^2(\mathcal{C})} = \|v\|_{L^2(\mathcal{G})}, \qquad
\|\tilde v\|_{H_0^\alpha(\mathcal{C})} = \|v\|_{H_0^\alpha(\mathcal{G})},
\]
so that the map \(v \mapsto \tilde v\) is an isometry from \(H_0^\alpha(\mathcal{G})\) into \(H_0^\alpha(\mathcal{C})\).
It follows that, for any \(v \in H_0^\alpha(\mathcal{G})\),
\[
\frac{\left( \int_{\mathcal{G}} v \right)^2}{\|v\|_{H_0^\alpha(\mathcal{G})}^2}
= \frac{\left( \int_{\mathcal{C}} \tilde v \right)^2}{\|\tilde v\|_{H_0^\alpha(\mathcal{C})}^2}
\leq \sup_{w \in H_0^\alpha(\mathcal{C})} \frac{\left( \int_{\mathcal{C}} w \right)^2}{\|w\|_{H_0^\alpha(\mathcal{C})}^2}
= T_{\alpha}(\mathcal{C}),
\]
and taking the supremum over all \(v\) gives the desired inequality
\[
T_{\alpha}(\mathcal{G}) \leq T_{\alpha}(\mathcal{C}).\qedhere
\]
\end{proof}
\noindent \textbf{Cutting a cycle to an interval.}
Cutting a cycle at one of its vertices produces an interval graph (see Figure \ref{fig:cutting}). It provides a bridge between graphs with cyclic symmetry and the fundamental case of an interval.
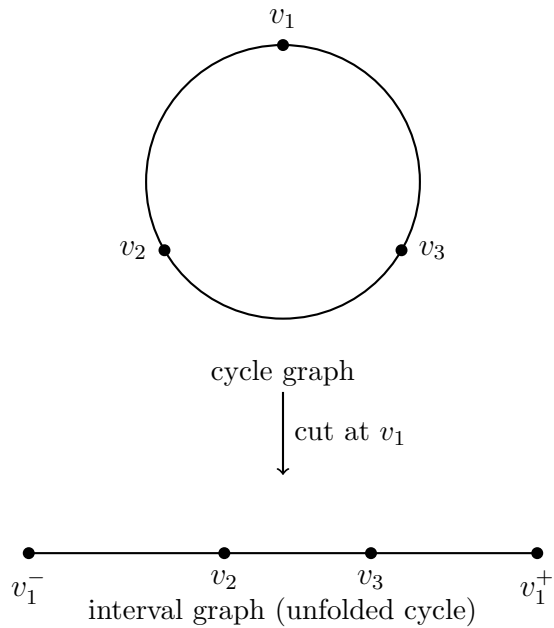
\begin{figure}[h!]
\centering
\begin{tikzpicture}[scale=1.3]
  \tikzset{vertex/.style={circle,fill=black,inner sep=1.6pt}}

  \def\R{1.4}
  \draw[thick] (0,0) circle (\R);

  \node[vertex,label=above:$v_1$] at (0,\R) {};
  \node[vertex,label=left:$v_2$]  at (-1.212,-0.70) {};
  \node[vertex,label=right:$v_3$] at ( 1.212,-0.70) {};

  \node at (0,-1.95) {\small cycle graph};

  \draw[->,thick] (0,-2.15) -- (0,-3.0) node[midway,right] {\small cut at $v_1$};

  \draw[thick] (-2.6,-3.8) -- (2.6,-3.8);

  \node[vertex,label=below:$v_1^{-}$] at (-2.6,-3.8) {};
  \node[vertex,label=below:$v_1^{+}$] at ( 2.6,-3.8) {};

  \node[vertex,label=below:$v_2$] at (-0.6,-3.8) {};
  \node[vertex,label=below:$v_3$] at ( 0.9,-3.8) {};

  \node at (0,-4.4) {\small interval graph (unfolded cycle)};
\end{tikzpicture}
\caption{Cutting the cycle at $v_1$ splits it into two boundary points $v_1^{-},v_1^{+}$ and unfolds the loop to an interval; interior vertices (e.g.\ $v_2,v_3$) lie along the resulting path.}
\label{fig:cutting}
\end{figure}

\begin{theorem}[Cycle Cutting to Interval]\label{princ:cut-cycle}
	Let $\alpha \in (0,1)$ and  \(\mathcal{C}_L\) be a metric cycle graph of total length \(L\) with at least one Dirichlet vertex \(v_0\). Cutting \(\mathcal{C}_L\) at \(v_0\) yields an interval graph \(\mathcal{I}_L\) of the same length with Dirichlet boundary conditions at both endpoints. Then
	\[
	T_{\alpha}(\mathcal{C}_L) =T_{\alpha}(\mathcal{I}_L).
	\]
\end{theorem}

\begin{proof}
	The solution \(u_\alpha\) to \((-\Delta_{\mathcal{C}_L})^\alpha u_\alpha = 1\) on \(\mathcal{C}_L\) with Dirichlet condition at \(v_0\) belongs to the space \(H_0^\alpha(\mathcal{C}_L)\) with \(u_\alpha(v_0) = 0\).
	Cutting the cycle at \(v_0\) produces an interval \(\mathcal{I}_L\), and any function \(u\) on \(\mathcal{C}_L\) vanishing at \(v_0\) naturally becomes a function on \(\mathcal{I}_L\) vanishing at both endpoints. Conversely, any function in \(H_0^\alpha(\mathcal{I}_L)\) extends uniquely to a function on \(\mathcal{C}_L\) vanishing at \(v_0\).
	
	Since the norms \(\|u\|_{{H}_0^\alpha}\) and the energy functional \(\int u\) are preserved under this identification, the variational problems defining \(T_\alpha\) on both graphs are equivalent. Thus,
	\[
	T_{\alpha}(\mathcal{C}_L) = T_{\alpha}(\mathcal{I}_L).\qedhere
	\]
\end{proof}

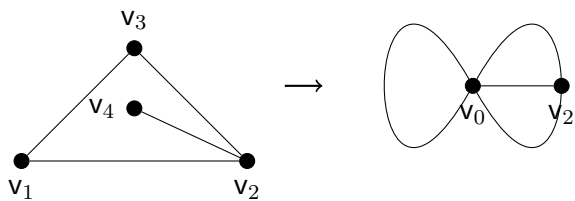
\begin{figure}
    \centering
    \begin{tikzpicture}
        \node[draw, circle, inner sep=2pt, fill,label=below:$\mathsf{v}_1$] (v1) at (0,0) {};
        \node[draw, circle, inner sep=2pt, fill, label=below:$\mathsf{v}_2$] (v2) at (3,0) {};
        \node[draw, circle, inner sep=2pt, fill,label=above:$\mathsf{v}_3$] (v3) at (1.5,1.5) {};
        \node[draw, circle, inner sep=2pt, fill,label=left:$\mathsf{v}_4$] (v4) at (1.5,0.7) {};
        
        \draw (v1) -- (v2);
        \draw (v1) -- (v3);
        \draw (v2) -- (v3);
        \draw (v4) -- (v2);

        \draw[thick,->] (3.5,1) -- (4,1);

        \node[draw, circle, inner sep=2pt, fill, label=below:$\mathsf{v}_0$] (v0) at (6,1) {};
        \node[draw, circle, inner sep=2pt, fill, label=below:$\mathsf{v}_2$] (v2p) at (7.18,1) {};
        
    \draw (v0) to[out=120,in=240,looseness=15,min distance=3cm] (v0);
    \draw (v0) to[out=60,in=300,looseness=15,min distance=3cm] (v0);
    \draw (v0) -- (v2p);
    \end{tikzpicture}
      \caption{The graph \( \mathcal{G} \) on the right is obtained from the graph on the left by gluing the vertices \( \mathsf{v}_1 \), \( \mathsf{v}_3 \), and \( \mathsf{v}_4 \) into a single vertex. Conversely, the graph on the left represents one possible reconstruction from the graph on the right by splitting \( \mathsf{v}_0 \) into the vertices \( \mathsf{v}_1 \), \( \mathsf{v}_3 \), and \( \mathsf{v}_4 \).}
      \label{fig:mergingandsplittingvertices}
\end{figure}
\noindent\textbf{Gluing vertices.}
Gluing vertices means identifying two or more vertices into a single one (see Figure \ref{fig:mergingandsplittingvertices}).
\begin{theorem}[Gluing Vertices Decreases Torsional Rigidity]\label{prin:gluingvertices}
	Let $\alpha \in (0,1)$ and \(\mathcal{G}\) be a compact connected metric graph, and let \(\mathcal{G}'\) be the graph obtained from \(\mathcal{G}\) by gluing two distinct vertices \(v_1, v_2 \in \mathcal{V}(\mathcal{G})\). Denote by \(T_{\alpha}(\mathcal{G})\) the torsional rigidity associated with the operator \((-\Delta_{\mathcal{G}})^\alpha\) under Dirichlet conditions on a fixed subset of vertices. Then 
	\[
	T_{\alpha}(\mathcal{G}') \leq T_{\alpha}(\mathcal{G}).
	\]
\end{theorem}
\begin{proof}
	When the two vertices \(v_1\) and \(v_2\) are glued, the resulting function space \(H_0^\alpha(\mathcal{G}')\) becomes a closed subspace of \(H_0^\alpha(\mathcal{G})\), since admissible functions on \(\mathcal{G}'\) are required to satisfy the additional constraint
	\[
	f(v_1) = f(v_2).
	\]
	This implies that
	\(
	H_0^\alpha(\mathcal{G}') \subset H_0^\alpha(\mathcal{G}),
	\)
	by \eqref{eqn:variationalcharacterization}, the supremum over the smaller space is less than or equal to that over the larger space:
	\[
	T_{\alpha}(\mathcal{G}') \leq T_{\alpha}(\mathcal{G}). \qedhere
	\]
\end{proof}
\section{Bounds on Torsional Rigidity}\label{section:bounds}
\noindent
Having established the variational characterization and surgery principles for fractional torsional rigidity, we now turn to general bounds for $T_\alpha(\mathcal{G})$ in terms of geometric quantities of the graph. 
The aim is to identify universal lower and upper bounds that depend only on simple features such as total length, number of edges, or boundary vertices. 
In this section, we combine the previously established surgery principles with explicit constructions to derive meaningful inequalities for the torsional rigidity.

\begin{theorem}\label{thm:lowerbound}
Let $\mathcal{G}$ be a compact, connected metric graph with total length $|\mathcal{G}|$ and with $|\mathcal{E}|$ edges. Fix 
\(
\alpha\in(0,1).
\)
Set 
\(
L:=\frac{|\mathcal{G}|}{|\mathcal{E}|},
\)
and let $\mathcal{F}$ be an equilateral flower graph with $|\mathcal{E}|$ petals (edges) each of length $L$ (hence $|\mathcal{F}|=|\mathcal{G}|$). Then the fractional torsional rigidity satisfies the lower bound
\[
T_\alpha(\mathcal{G}) \;\ge\; T_\alpha(\mathcal{F})
\;=\; |\mathcal{E}|\; 8\,L^{1+2\alpha}\,\pi^{-2(1+\alpha)}\Big(1-2^{-2(1+\alpha)}\Big)\,\zeta\big(2(1+\alpha)\big).
\]
\end{theorem}

\begin{proof}
By identifying (gluing) all internal vertices appropriately, $\mathcal{G}$ can be transformed into a flower graph $\mathcal{F}$ with the same number of edges and total length. 
By Theorem \ref{prin:gluingvertices}, one obtains 
\[
T_\alpha(\mathcal{G}) \;\ge\; T_\alpha(\mathcal{F}).\qedhere
\]
\end{proof}

\begin{theorem}
Let $\mathcal{G}$ be a compact, connected metric graph of total length $|\mathcal{G}|$, equipped with Dirichlet conditions on a nonempty subset $\mathcal{V}_D$ of its vertices. 
Fix $\alpha \in (0,1)$. 
Then the fractional torsional rigidity satisfies the upper bound
\[
T_{\alpha}(\mathcal{G}) \;\le\; T_\alpha(\mathcal{I})=\frac{8\,2^{2\alpha}\,|\mathcal{G}|^{2\alpha+1}}{\pi^{2+2\alpha}} \left(1 - 2^{-(2+2\alpha)}\right)\zeta(2+2\alpha),
\]
where $\mathcal{I}$ is the interval graph $[0,|\mathcal{G}|]$ endowed with Dirichlet boundary at one endpoint and Neumann boundary at the other. 
\end{theorem}


\begin{proof}
To establish the upper bound, we apply the surgery principles sequentially. First, by the edge-doubling principle (Theorem~\ref{princ:edge-doubling}), we have
\[
T_{\alpha}(\mathcal{G}) \leq \frac{1}{2} \, T_{\alpha}(\tilde{\mathcal{G}}),
\]
where $\tilde{\mathcal{G}}$ denotes the graph obtained by doubling each edge of $\mathcal{G}$.  
Next, applying the unfolding principle (Theorem~\ref{princ:unfold-cycle}) to $\tilde{\mathcal{G}}$ yields a cycle graph $\mathcal{C}_{2|\mathcal{G}|}$ of length $2|\mathcal{G}|$, satisfying
\[
T_\alpha(\mathcal{G})\leq \frac{1}{2}T_{\alpha}(\tilde{\mathcal{G}}) \le \frac{1}{2}T_{\alpha}(\mathcal{C}_{2|\mathcal{G}|}).
\]
Then, by cutting the cycle at a Dirichlet vertex using Theorem~\ref{princ:cut-cycle}, we obtain an interval graph $\mathcal{I}_{2|\mathcal{G}|}$ with Dirichlet endpoints, so that
\[
T_{\alpha}(\mathcal{C}_{2|\mathcal{G}|}) = T_{\alpha}(\mathcal{I}_{2|\mathcal{G}|}).
\]
Computing the fractional torsional rigidity of the interval graph with two Dirichlet end points and using Example \ref{ex: interval}, we have 
\[
T_{\alpha}(\mathcal{I}_{2|\mathcal{G}|}) =   \frac{8  }{\pi^{2+2\alpha}}
\Big(1 - 2^{-(2+2\alpha)}\Big)\zeta(2+2\alpha)  \, |2\mathcal{G}|^{1+2\alpha}=2T_{\alpha}([0,|\mathcal{G}|]),
\]
where the interval $[0,|\mathcal{G}|]$ is endowed with Dirichlet boundary at one endpoint and Neumann boundary at the other. 
Combining these steps, we conclude
\[
T_{\alpha}(\mathcal{G}) \leq \frac{1}{2}  T_\alpha(\tilde{G}) \leq \frac{1}{2}T_\alpha (\mathcal{I}_{2|\mathcal{G}|})
= T_\alpha([0,|\mathcal{G}|]).\qedhere
\]
\end{proof}


\begin{thebibliography}{99}

\bibitem{brasco}
L.~Brasco,
On torsional rigidity and principal frequencies: an invitation to the Kohler--Jobin rearrangement technique,
{\em ESAIM Control Optim. Calc. Var.} 20(2):315--338, 2014.

\bibitem{BK}
G.~Berkolaiko and P.~Kuchment,
{\em Introduction to Quantum Graphs},
Mathematical Surveys and Monographs, vol.~186,
American Mathematical Society, Providence, RI, 2013.

\bibitem{colla}
D.~Colladay, L.~Kaganovskiy, and P.~McDonald,
Torsional rigidity, isospectrality and quantum graphs,
{\em J. Phys. A} 50:035201, 2016.

\bibitem{DuefelKennedyMugnoloPluemerTaeufer-22}
M.~Düfel, J.~B.~Kennedy, D.~Mugnolo, M.~Plümer, and M.~Täufer,
Boundary conditions matter: on the spectrum of infinite quantum graphs,
arXiv:2207.04024 [math.SP], 2022.

\bibitem{dunford}
N.~Dunford and J.~T.~Schwartz,
{\em Linear Operators, Part 2: Spectral Theory, Self Adjoint Operators in Hilbert Space},
Wiley Classics Library, Wiley, 1988.

\bibitem{Evans}
L.~C.~Evans,
{\em Partial Differential Equations},
2nd ed.,
Graduate Studies in Mathematics, vol.~19,
American Mathematical Society, Providence, RI, 2010.

\bibitem{Faber-23}
G.~Faber,
Beweis, dass unter allen homogenen Membranen von gleicher Fläche und gleicher Spannung die kreisförmige den tiefsten Grundton gibt,
{\em Sitzungsber. Bayer. Akad. Wiss. München, Math.-Phys. Kl.}, 169--172, 1923.

\bibitem{GluckThesis}
J.~Glück,
{\em Invariant Sets and Long Time Behaviour of Operator Semigroups},
PhD thesis, Universität Ulm, 2016.

\bibitem{Karreskog2015}
G.~Karreskog, P.~Kurasov, and K.~I.~Trygg,
Schr{\"o}dinger operators on graphs: symmetrization and Eulerian cycles,
Research Reports in Mathematics, Stockholm University, Department of Mathematics, 3:14, 2015.

\bibitem{KennedyMugnoloTaeufer-2024_preprint}
J.~B.~Kennedy, D.~Mugnolo, and M.~Täufer,
Towards a theory of eigenvalue asymptotics on infinite metric graphs: the case of diagonal combs,
arXiv:2403.10708 [math.SP], 2024.

\bibitem{kost1}
V.~Kostrykin, J.~Potthoff, and R.~Schrader,
Laplacians on metric graphs: eigenvalues, resolvents and semigroups,
in {\em Quantum Graphs and Their Applications}, Contemp. Math. 415,
Amer. Math. Soc., Providence, RI, 201--225, 2006.

\bibitem{kost2}
V.~Kostrykin, J.~Potthoff, and R.~Schrader,
Contraction semigroups on metric graphs,
in {\em Analysis on Graphs and Its Applications},
P.~Exner et al. (eds.),
Proc. Symp. Pure Math. 77,
Amer. Math. Soc., Providence, RI, 423--458, 2008.

\bibitem{KostenkoMN-22}
A.~Kostenko, D.~Mugnolo, and N.~Nicolussi,
Self-adjoint and Markovian extensions of infinite quantum graphs,
{\em J. Lond. Math. Soc.} 105(2):1262--1313, 2022.

\bibitem{KostenkoN-19}
A.~Kostenko and N.~Nicolussi,
Spectral estimates for infinite quantum graphs,
{\em Calc. Var.} 58:15, 2019.

\bibitem{Krahn-25}
E.~Krahn,
Über eine von Rayleigh formulierte Minimaleigenschaft des Kreises,
{\em Math. Ann.} 94:97--100, 1925.

\bibitem{kuch1}
P.~Kuchment,
Quantum graphs: I. Some basic structures,
{\em Waves Random Media} 14(1):S107--S128, 2004.

\bibitem{kuch2}
P.~Kuchment,
Quantum graphs: an introduction and a brief survey,
in {\em Analysis on Graphs and Its Applications},
P.~Exner et al. (eds.),
Proc. Symp. Pure Math. 77,
Amer. Math. Soc., Providence, RI, 291--314, 2008.

\bibitem{lischke}
A.~Lischke, G.~Pang, M.~Gulian, F.~Song, C.~Glusa, X.~Zheng, Z.~Mao, W.~Cai,
M.~M.~Meerschaert, M.~Ainsworth, and G.~E.~Karniadakis,
What is the fractional Laplacian? A comparative review with new results,
{\em J. Comput. Phys.} 404:109009, 2020.

\bibitem{mug3}
D.~Mugnolo,
What is actually a metric graph?,
arXiv:1912.07549 [math.CO], 2019.

\bibitem{mugnolo2014}
D.~Mugnolo,
{\em Semigroup Methods for Evolution Equations on Networks},
Springer-Verlag, Berlin, 2014.

\bibitem{MugnoloPlumer2023}
D.~Mugnolo and M.~Plümer,
On torsional rigidity and ground-state energy of compact quantum graphs,
{\em Calc. Var. Partial Differential Equations} 62:27, 2023.
\url{https://link.springer.com/article/10.1007/s00526-022-02363-9}.

\bibitem{nica}
S.~Nicaise,
Spectre des r{\'e}seaux topologiques finis,
{\em Bull. Sci. Math.} II. Sér. 111:401--413, 1987.

\bibitem{OzcanTaeufer2024}
S.~Özcan and M.~Täufer,
Torsional rigidity on metric graphs with delta-vertex conditions,
arXiv:2410.18545 [math.SP], 2024.
\url{https://arxiv.org/abs/2410.18545}.

\bibitem{pol1}
G.~P{\'o}lya,
Torsional rigidity, principal frequency, electrostatic capacity and symmetrization,
{\em Quart. Appl. Math.} 6(3):267--277, 1948.

\bibitem{pol2}
G.~P{\'o}lya and G.~Szeg\H{o},
{\em Isoperimetric Inequalities in Mathematical Physics},
Annals of Mathematics Studies 27,
Princeton University Press, Princeton, 1951.

\bibitem{pol3}
G.~P{\'o}lya and A.~Weinstein,
On the torsional rigidity of multiply connected cross-sections,
{\em Ann. Math.} (2) 52(1):154--163, 1950.

\end{thebibliography}
\end{document}